\documentclass[11pt]{amsart}
\usepackage{amsmath,amsfonts,amsthm,amssymb,amscd,verbatim,amsrefs, hyperref}

\numberwithin{equation}{section}

\newtheorem{Theorem}[equation]{Theorem}
\newtheorem{Conjecture}[equation]{Conjecture}

\newtheorem{Lemma}[equation]{Lemma}

\theoremstyle{remark}
\newtheorem*{Remark}{Remark}

\newtheorem{Question}[equation]{Question}

\newcommand{\abs}[1]{\left\lvert#1\right\rvert}
\newcommand{\col}{\,{:}\,}

\newcommand{\R}{\ensuremath{\mathbb{R}}}
\newcommand{\C}{\ensuremath{\mathbb{C}}}

\newcommand{\F}{\mathbb{F}}

\title[the number of rich lines in high dimensional real vector spaces]{On the number of rich lines in high dimensional real vector spaces}
\author{M\'arton Hablicsek}
\address{
  Department of Mathematics,
  Leiden University,
  Snellius Building, 
  Niels Bohrweg 1, 
  2333 CA Leiden,
  Netherlands
}
\email{hablicsekhm@math.leidenuniv.nl}
\author{Zachary Scherr}
\address{
  Department of Mathematics and Computer Science,
  Fisher Hall
  Susquehanna University,
  514 University Ave.,
  Selinsgrove, Pa. 17870 USA
}
\email{scherr@susqu.edu}
\date{\today}
\begin{document}

\begin{abstract}In this short note we use the Polynomial Ham Sandwich Theorem to strengthen a recent result of Dvir and Gopi about the number of rich lines in high dimensional Euclidean spaces. Our result shows that if there are sufficiently many rich lines incident to a set of points then a large fraction of them must be contained in a hyperplane.\end{abstract}
\maketitle

\section{Introduction}
Let $P$ be a set of points of size $n$ in $\R^d$, and consider a set of lines $L$ in $\R^d$ so that each line in $L$ contains at least $r$ points of $P$.  We investigate the possible size of $L$.

We begin our discussion with the case of $d=2$.  The celebrated result of Szemer\'edi and Trotter (which was generalized to the complex plane by T\'oth \cite{Tot} and Zahl \cite{Zah1}) asserts the following.
 
\begin{Theorem}[\cite{SzeTro}]
Given $P$, a set of points in $\R^2$, and $L$, a set of lines, the number of incidences $I(L,P)$ between $L$ and $P$ satisfies
\[I(L,P)=O(\abs{L}^{2/3}\abs{P}^{2/3}+\abs{L}+\abs{P}).\]
\end{Theorem}
\medskip

In our case, each line contains at least $r$ points of $P$, therefore, $I(L,P)\geq r\abs{L}$.  Rearranging the terms we obtain that 
\[\abs{L}=O\left(\frac{n^2}{r^3}+\frac{n}{r}\right).\]
This bound is sharp. In a 2-dimensional square grid of $n$ points, for example, each line parallel to one of  the sides of the square contains $O(\sqrt{n})$ points, and there are $O(\sqrt{n})=O\left(\frac{n^2}{(\sqrt{n})^3}\right)$ such lines.

In the higher dimensional case, the $d$-dimensional grid of $n$ points contains $O\left(\frac{n^2}{r^{d+1}}\right)$ lines for $r=o(n^{1/d})$ \cite{SolVu}.  Similar constructions can be given using low dimensional grids as well.  Motivated by these examples, Dvir and Gopi conjectured the following.

\begin{Conjecture}[\cite{DviGop}]\label{DviGopConj}
Let $P$ be a set of $n$ points in $\C^d$ and let $L$ be a set of lines so that each line contains at least $r$ points of $P$. There are constants $K$ and $N$, dependent only on $d$, so that if 
\[\abs{L}\geq K\left(\frac{n^2}{r^{d+1}}+\frac{n}{r}\right),\]
then there exists $1<\ell<d$ and a subset $P'\subseteq P$ of size $N\frac{n}{r^{d-\ell}}$ which is contained in an $\ell$-dimensional affine subspace.
\end{Conjecture}
\medskip

In their paper \cite{DviGop}, Dvir and Gopi show a weaker version of the conjecture.

\begin{Theorem}[\cite{DviGop}]
\label{thm:dvi}
Let $P$ be a set of $n$ points in $\C^d$ and let $L$ be a set of lines so that each line contains at least $r$ points of $P$. There are constants $K$ and $N$, dependent only on $d$, so that if
\[\abs{L}\geq K\frac{n^2}{r^d}\]
then there exists a subset of $P$ of size $N\frac{n}{r^{d-2}}$ contained in a $(d-1)$-dimensional hyperplane.
\end{Theorem}  
\medskip

Their proof involves a clever use of design matrices in order to show that almost
all the lines lie in a low degree hypersurface (the degree needs to be less than $r$).  In our paper, we prove a stronger version of Theorem~\ref{thm:dvi} but over $\R$ rather than over $\C$.  The strategy of our
proof is similar to that of Dvir and Gopi, except working over $\R$ allows us to use the Polynomial Ham Sandwich Theorem (Theorem~\ref{thm:phs}) in place of design matrices.

\section{Main results}

Our main result shows that if there are too many $r$-rich lines then most of the lines must lie in a low degree hypersurface.

\begin{Theorem}
\label{thm:lines}
Let $P$ be a set of $n$ points in $\R^d$ and let $L$ be a set of lines so that each line contains at least $r$ points of $P$. There is a constant $K$, dependent only on $d$, so that if 
\[\abs{L}\geq K\frac{n^2}{r^{d+1}},\]
then there exists a hypersurface of degree at most $\frac{r}{4}$ containing at least $4\frac{n^2}{r^{d+1}}$ lines of $L$.
\end{Theorem}

\begin{Remark} 
One can interpret the theorem above as follows.  If a set of points is such that there exist a lot of non-generic large subsets, then a large fraction of the points must be non-generic.  In our case we know that there are $K\frac{n^2}{r^{d+1}}$ non-generic subsets of size $r$, and we deduce that a large fraction of points lie in a low degree hypersurface.
\end{Remark}

As an easy consequence of Theorem \ref{thm:lines} we obtain a better bound over $\R$ than the bound in Theorem \ref{thm:dvi}.

\begin{Theorem}
\label{thm:main}
Let $P$ be a set of $n$ points in $\R^d$ and let $L$ be a set of lines so that each line contains at least $r$ points of $P$. There are constants $K$ and $N$, dependent only on $d$, so that if 
\[\abs{L}\geq K\frac{n^2}{r^{d+1}},\]
then there exists a hyperplane containing $N\frac{n}{r^{d-1}}$ points of $P$.
\end{Theorem}
\medskip

The main technique in our proof is the Polynomial Ham Sandwich Theorem which we state below.

\begin{Theorem}[Polynomial Ham Sandwich]
\label{thm:phs}
Let $S$ be a finite set of points in $\R^d$, and let $m \geq 1$. Then there exists a non-trivial polynomial $f$ of degree $m$ and a decomposition of $\{ x \in \R^d\col f(x) \neq 0 \}$ into at most $O(m^d)$ cells each of which is an open set with boundary in $\{ x \in \R^d\col f(x) = 0 \}$, and each of which contains at most $O\left(\frac{\abs{S}}{m^d}\right)$ points of $S$.
\end{Theorem}
\medskip

This poweful tool was invented by Guth and Katz in \cite{GutKat} to give a nearly complete solution to the Erd\H{o}s distinct distance problem and has been applied, for instance, to give a new proof of the Szemer\'edi-Trotter theorem, the Pach-Sharir theorem \cite{PacSha} (see \cite{KapMatSha} for more details) and some variants of the joints problem \cite{Ili}.

We remark that the Polynomial Ham Sandwich Theorem relies on the topology of $\R$, and thus our proof only works over $\R$.  On the other hand, we believe that Theorem \ref{thm:lines} holds over any prime field $\F_p$ and over $\C$ as well, and hence it would be nice to see a proof of Theorem \ref{thm:lines} which does not use the Polynomial Ham Sandwich Theorem.

\begin{Question}
Let $P$ be a set of $n$ points in $k^d$, where $k$ is either a prime field $\F_p$ or the field of complex numbers. Let $L$ be a set of lines so that each line contains at least $r$ points of $P$. Is there a constant $K$, dependent only on $d$, so that if 
\[\abs{L}\geq K\frac{n^2}{r^{d+1}},\]
then there exists a hypersurface of degree at most $\frac{r}{4}$ containing at least $4\frac{n^2}{r^{d+1}}$ lines of $L$?
\end{Question}

We remark that recently in \cite{Zah2}, Zahl proved a slightly weaker version of Theorem \ref{thm:main} over $\C$ using a version of the Polynomial Ham Sandwich Theorem over $\C$ (see \cite{Tot} or \cite{Zah1}).

\section{Proof of the main theorems}

In this section we prove Theorems \ref{thm:lines} and \ref{thm:main}.  We begin with the proof of Theorem \ref{thm:lines} which we restate below.

\begin{Theorem}
Let $P$ be a set of $n$ points in $\R^d$, and let $L$ be a set of lines so that each line contains at least $r$ points of $P$. There is a constant $K$, dependent only on $d$, so that if 
\[\abs{L}\geq K\frac{n^2}{r^{d+1}},\]
then there exists a hypersurface of degree at most $\frac{r}{4}$ containing at least $4\frac{n^2}{r^{d+1}}$ lines of $L$.
\end{Theorem}

\begin{proof}
Assume that $\abs{L}=K\frac{n^2}{r^{d+1}}$ for a large constant $K$ (which will be chosen in the end of the proof) and
fix a positive integer $m$ in the range $\frac{r}{8} < m < \frac{r}{4}$ (the interesting case of the theorem is when $r$ is large).
Using the Polynomial Ham Sandwich Theorem (Theorem \ref{thm:phs}), we can find a polynomial $f$ of degree $m$ partitioning
$\R^d$ into the zero locus of $f$ as well as $M=O(m^d)$ open cells
\[\R^d = \{x\col f(x)=0\}\cup C_1\cup C_2\cup\ldots\cup C_M\]
so that each cell contains at most $O\left(\frac{n}{m^d}\right)$ points of $P$ and has boundary in the zero set of $f$.  We denote $P_i:=C_i\cap P$.

Let
\[L_{cell} = \{\ell\in L\col\exists i\text{ with }\abs{\ell\cap P_i}\ge 2\}.\]
Since the zero locus of $f$ forms the boundary of the union of the cells, B\'ezout's theorem guarantees that every line in $\R^d$ intersects at most $m$ cells.  If $\ell\in L\backslash L_{cell}$, then $\abs{\ell\cap P_i}\le 1$ for each $i$, so in particular
\begin{equation}\label{estimate1}
\abs{\bigcup _{i=1}^M \ell\cap P_i}=\sum_{i=1}^M\abs{\ell\cap P_i} \le m < \frac{r}{2}.
\end{equation}
By assumption, every line in $L$ contains $r$ points of $P$ so lines in $L\backslash L_{cell}$ must contain at least $\frac{r}{2} > m$ in the zero locus of $f$. We can again invoke B\'ezout to conclude that such a line is necessarily contained in the zero locus of $f$.  Since what we are after is a lower bound on the number of lines in $L$ which are contained in the zero locus of $f$, this discussion shows that it suffices to give an upper bound on the size of $L_{cell}$.

To do so, we take advantage of the fact that every line $\ell\in L_{cell}$ has the property that $\abs{\ell\cap P_i}\ge 2$ for some $i$.  The total number of lines, counted with multiplicity, in $\R^d$ which intersect some $P_i$ in at least two points is
\begin{equation}\label{estimate0}
\sum_{i=1}^M\binom{\abs{P_i}}{2}\end{equation}
where each such line $\ell$ is counted with multiplicity
\[k_\ell:=\sum_{i=1}^M\binom{\abs{\ell\cap P_i}}{2} =\frac{1}{2}\left(\sum_{i=1}^M \abs{\ell\cap P_i}^2-\sum_{i=1}^M \abs{\ell\cap P_i}\right).\]
We have already observed that a line not contained in the zero locus of $f$ can only intersect at most $m$ cells.  If
\[a_i=\begin{cases}
0, & \ell\cap P_i=\emptyset\\
1, & \text{otherwise}\end{cases}\]
then this observation, combined with the the Cauchy-Schwarz inequality, gives
\begin{eqnarray*}
\left(\sum_{i=1}^M \abs{\ell\cap P_i}\right)^2&=&
\left(\sum_{i=1}^M a_i\abs{\ell\cap P_i}\right)^2\\
&\le& \sum_{i=1}^M a_i^2\cdot\sum_{i=1}^M \abs{\ell\cap P_i}^2\\
&\le& m\sum_{i=1}^M \abs{\ell\cap P_i}^2.\end{eqnarray*}
Therefore we get a lower bound
\begin{equation}\label{estimate2}
k_\ell \ge \frac{1}{2}\left(\frac{\left(\sum_{i=1}^M \abs{\ell\cap P_i}\right)^2}{m}-\sum_{i=1}^M \abs{\ell\cap P_i}\right).
\end{equation}
If $\ell\in L_{cell}$, then (\ref{estimate1}) guarantees that
\[\sum_{i=1}^M\abs{\ell\cap P_i}\ge\frac{r}{2}.\]
For such $\ell$, (\ref{estimate2}) becomes
\[k_\ell\ge\frac{r}{4}\left(\frac{r}{2m}-1\right) = \frac{r^2-2mr}{8m}.\]
Since $m < \frac{r}{4}$, it follows that
\[k_\ell\ge \frac{r^2-r^2/2}{8m}\ge \frac{r^2}{16m}\]
when $r$ is large enough.

Every $\ell\in L_{cell}$ is counted with multiplicity $k_\ell$ in (\ref{estimate0}).  Thus
\begin{equation}\label{estimate3}
\sum_{i=1}^M\binom{\abs{P_i}}{2}\ge\sum_{\ell\in L_{cell}} k_\ell\ge\abs{L_{cell}}\frac{r^2}{16m}.\end{equation}
We know that $M=O(m^d)$ and $\abs{P_i} = O\left(\frac{n}{m^d}\right)$, so we can rewrite (\ref{estimate3}) as
\[\abs{L_{cell}} = \frac{16m}{r^2}O\left(m^d\frac{n^2}{m^{2d}}\right) = \frac{1}{r^2}O\left(\frac{n^2}{m^{d-1}}\right).\]
Since $\frac{r}{8} < m$, this last equation becomes
\[\abs{L_{cell}} = O\left(\frac{n^2}{r^{d+1}}\right).\]
The set of lines in $L$ which are contained in the zero locus of $f$ has size
\[\abs{L} - \abs{L_{cell}}\ge K\frac{n^2}{r^{d+1}} - \abs{L_{cell}},\]
and so we can choose $K$ large enough so as to ensure that this last quantity is bounded below by $4\frac{n^2}{r^{d+1}}$.
\end{proof}

As an easy corollary we prove Theorem \ref{thm:main}.  In order to do so, we use the following standard graph theoretic lemma which can also be found in the paper of Dvir and Gopi.
\begin{Lemma}[Lemma 2.8, \cite{DviGop}]
\label{lem:ave}
Let $G=(A\sqcup B,E)$ be a bipartite graph with a non-empty edge set $E\subset A\times B$.  Then there exist non-empty subsets $A'\subset A$ and $B'\subset B$ such that if we consider the induced subgraph $G'=(A'\sqcup B',E')$, then
\begin{itemize}
\item The minimum degree in $A'$ is at least $\frac{\abs{E}}{4\abs{A}}$,
\item The minimum degree in $B'$ is at least $\frac{\abs{E}}{4\abs{B}}$,
\item $\abs{E'}\geq \abs{E}/2$.
\end{itemize}
\end{Lemma}

We are ready to prove the theorem.

\begin{Theorem}
Let $P$ be a set of $n$ points in $\R^d$, and let $L$ be a set of lines so that each line contains at least $r$ points of $P$. There are constants $K$ and $N$, dependent only on $d$, so that if 
\[\abs{L}\geq K\frac{n^2}{r^{d+1}},\]
then there exists a hyperplane containing $N\frac{n}{r^{d-1}}$ points of $P$.
\end{Theorem}

\begin{proof}
We may use the previous theorem to conclude that if $K$ is large enough then there exists at least $4\frac{n^2}{r^{d+1}}$ lines contained in a degree $m<\frac{r}{4}$ hypersurface.  Let us denote the set of these lines by $L_Z$ and the set of points of $P$ on the lines of $L_Z$ by $P_Z$.  Each line of $L_Z$ is still $r$-riched, thus the total number of incidences between $L_Z$ and $P_Z$ satisfies
\[I(L_Z,P_Z)\ge r\abs{L_Z} = 4\frac{n^2}{r^{d}}.\]
By Lemma \ref{lem:ave} we may, after removing lines and points, therefore assume without loss of generality that each point of $P_Z$ is incident to at least $\frac{n}{r^{d}}$ lines in $L_Z$.

Let $g$ be a non-zero polynomial of minimum degree vanishing on $L_Z$.  We know that $f$ vanishes on $L_Z$, therefore the degree of $g$ is less than $r$.

Now, we call a point $p\in P_Z$ a {\it joint} if the directions of the lines in $L_Z$ incident to $p$ span $\R^d$.  If every $p\in P_Z$ is a joint, then surely the gradient of $g$ must vanish on all of $P_Z$.  Pick a component of the gradient which is non-zero on the vanishing 
locus of $g$.  This component vanishes on all the points in $P_Z$ and is of degree less than $r$.  Therefore, by B\'ezout's theorem, this component vanishes on all the lines in $L_Z$ as well, but the component is of smaller degree than of $g$ which is a contradiction.

Thus there must be a point $p\in P_Z$ which is not a joint, whence all the lines of $L_Z$ going through $p$ lie in the same hyperplane.  We know that there are $\frac{n}{r^d}$ lines going through $p$, and on each such line there are $r-1$ other points, implying that there are at least
\[(r-1)\frac{n}{r^d}+1 = \Omega\left(\frac{n}{r^{d-1}}\right)\]
points in one hyperplane. 
\end{proof}

\end{document}